\documentclass{amsart}
\usepackage{amssymb,amsmath, amsthm,latexsym}
\usepackage{graphics}
\usepackage{psfrag}
\usepackage{tikz-cd}
\usepackage[utf8]{inputenc}
\usepackage{amscd}
\usepackage{graphicx}
\newcommand{\cal}[1]{\mathcal{#1}}
\theoremstyle{plain}
\newtheorem*{theo}{Theorem}

\newtheorem*{prop}{Proposition}
\newtheorem*{conj}{Conjecture}
\newtheorem{lemma}{Lemma}[section]

\newtheorem{proposition}[lemma]{Proposition}
\newtheorem{corollary}[lemma]{Corollary}
\theoremstyle{definition}

\newtheorem{remark}[lemma]{Remark}
\let\egthree=\phi
\let\phi=\varphi
\let\varphi=\egthree

\begin{document}
\title{Some topological properties of surface bundles}
\author{Ursula Hamenst\"adt}
\thanks{Partially supported by the Hausdorff Center Bonn}
\date{August 10, 2020}


\begin{abstract}
  We describe the second integral cohomology group
  of a surface bundle as the group of Chern classes
  of fiberwise holomorphic complex line bundles and use this to obtain
  some new information on this group.
\end{abstract}

\maketitle

\section{Introduction}

A \emph{surface bundle over a surface} is a smooth 
closed 4-manifold $E$ which fibers over a closed oriented surface $B$ 
of genus $h\geq 0$,
with fiber a closed oriented surface $S_g$ of genus $g\geq 0$. 

Assume that $g,h\geq 2$. Then the monodromy of the
bundle determines a homomorphism $\rho$ of the fundamental
group $\pi_1(B)$ of $B$ into the 
\emph{mapping class group
$\Gamma_g$} of $S_g$, that is, the group of isotopy
classes of orientation preserving diffeomorphisms of $S_g$. Thus 
the geometry and topology of surface bundles over surfaces is intimately
related to properties of the mapping class group. 

Natural topological invariants of such 
surface bundles $E$ 
are the \emph{Euler characteristic} $\chi(E)$ 
and the \emph{signature} $\sigma(E)$.  
For the Euler
characteristic we have
\[\chi(E)=\chi(B)\chi(S_g)=(2h-2)(2g-2).\]
The signature can be computed as follows. 

The tangent bundle $\nu$ 
of the fibers of the surface bundle,
called the \emph{vertical tangent bundle}
in the sequel, is a real two-dimensional 
oriented smooth subbundle of the tangent bundle $TE$ of $E$. Hence it 
can be equipped with the structure of a complex line bundle.

Choose a smooth Riemannian metric on 
$TE$ and let 
$\nu^\perp$ be the orthogonal complement of $\nu$ in $TE$ for this
metric. Then the projection $\Pi:E\to B$ maps each fiber of $\nu^\perp$
isomorphically onto a fiber of $TB$ and hence as a smooth vector bundle,
$\nu^\perp$ is isomorphic to the bundle $\Pi^*TB$. 
Now $TB$ can be equipped with the structure of a complex line bundle
as well, and 
\[TE=\nu\oplus \Pi^*TB\]
is the direct sum of two complex line bundles (this is a decomposition
of $TE$ as a smooth vector bundle).   
In particular, the first and
second Chern class of $TE$ are defined,
and they are independent of the choices made.  

By Hirzebruch's signature theorem, the signature $\sigma(E)$ of 
$E$ equals 
\[\sigma(E)=\frac{1}{3}p_1(E)\]
where $p_1(E)$ is the first Pontryagin number of $E$. 
We then have 
\begin{equation}\label{signature}
\sigma(E)=\frac{1}{3}(c_1(E)^2-2c_2(E))=
\frac{1}{3}(c_1(E)^2-2\chi(E))\end{equation}
where as is customary, $c_1(E)^2$ and $c_2(E)$ denote Chern numbers
of $E$. 

As $TE=\nu\oplus \Pi^*TB$, the total Chern class of 
$TE$ equals 
\begin{align}
c(TE) &=(1+c_1(\nu))\cup (1+c_1(\Pi^*TB))\notag \\
 &=
1+c_1(\nu)+c_1(\Pi^*TB)+c_1(\nu)\cup c_1(\Pi^*TB)\notag
\end{align}
and hence sicne $c_1(\Pi^*TB)\cup c_1(\Pi^*TB)(E)=0$ we have 
\begin{equation}\label{prod}
3\sigma(E)=c_1(\nu)\cup c_1(\nu)(E).\end{equation}

Let ${\cal M}_g$ be the \emph{moduli space of complex curves of genus $g$}, that,
is the moduli space of complex structures on $S_g$ up to biholomorphic 
equivalence, 
and let ${\cal U}\to {\cal M}_g$ be the \emph{universal curve} whose 
fiber over a point $X\in {\cal M}_g$ is just the complex curve $X$. 
Let us now assume that $E$ is a \emph{Kodaira fibration}, that is, $E$ is 
a complex manifold, and the complex structures of the fibers vary nontrivially.  
This is equivalent to stating that there is a complex structure on the base $B$, and
there is 
a nonconstant holomorphic map $\phi:B\to {\cal M}_g$
 such that $E=\phi^*{\cal U}$.
 By the classification of complex surfaces, a Kodaira fibration is of general type
and hence projective and K\"ahler. 

The Miaoka inequality for complex surfaces $Y$ of general type 
states that 
$c_1^2(Y)\leq 3c_2(Y)$. 
Therefore by equation (\ref{signature}) which is valid for all 
closed oriented 4-manifolds, we have 
\[3\vert \sigma(Y)\vert \leq \vert \chi(Y)\vert.\]
Equality holds if and only if $Y$ is a quotient of the ball.
On the other hand, Kapovich \cite{Ka98} showed that  
no surface bundle over a surface
is a quotient of the ball and hence we always have
$3\vert \sigma(E)\vert < \vert \chi(E)\vert $ for all 
Kodaira fibrations $E$. It is also known that the 
signature of a Kodaira fibration does not vanish. 

For surface bundles over surfaces
which do not admit a complex structure, much less is known
about the relation between signature and Euler characteristic. 
The most general result to date seems to be  
 a theorem of Kotschick \cite{K98}. Using Seibert Witten invariants, 
 he showed 
 \begin{equation}\label{kotschick}
 2\vert \sigma(E)\vert \leq \vert \chi(E)\vert \end{equation}
for all surface bundles over surfaces. 
If $E$ admits an Einstein metric,
then the stronger inequality
\[3\vert \sigma(E)\vert <\vert   \chi(E)\vert \]
holds true, generalizing the Miaoka inequality for Kodaira fibrations. 
The following conjecture was formulated among others in \cite{K98}.

\begin{conj}\label{conjecture}
$3\vert \sigma(E)\vert \leq \vert \chi(E)\vert$ holds true for all surface
bundles over surfaces. 
\end{conj}

Perhaps the motivation for this conjecture stems from the general belief that
aspherical smooth closed 4-manifolds should admit an Einstein metric. 

The conjecture can be 
viewed as a twisted higher dimensional version
of the Milnor-Wood inequality for the Euler number of a flat
circle bundle over a closed oriented surface. 
Namely, call 
a circle bundle $H\to M$ 
over a manifold $M$ (or any CW-complex) 
\emph{flat} 
\cite{M58,W71} 
if the following holds true. Let ${\rm Top}^+(S^1)$ be
the group of orientation preserving homeomorphisms
of the circle $S^1$. We require that there is a homomorphism
$\eta:\pi_1(M)\to {\rm Top}^+(S^1)$ such that
\[H=\tilde M\times S^1/\pi_1(M)\]
where $\tilde M$ is the universal covering of $M$ and 
$\pi_1(M)$ acts on
$\tilde M\times S^1$ via
\[(x,t)g=(xg,\eta(g)^{-1}(t)).\]
 The same definition applies if $M$ is a good 
orbifold, that is, $M$ is the quotient of a
smooth simply connected manifold $\tilde M$ by 
the action of a group of 
diffeomorphisms which acts properly discontinuously,
but not necessarily freely.

The celebrated Milnor Wood inequality
bounds the absolute value of the Euler number (or first Chern class) 
of a complex line bundle with flat circle 
subbundle over a closed surface by the absolute value of the Euler characteristic
of the surface \cite{M58,W71}.

Now as was pointed out by Morita \cite{Mo88}, 
the circle subbundle of the vertical tangent bundle of a
surface bundle $\Pi:E\to B$ over an arbitrary smooth base $B$ is flat.
In this vein, the conjecture predicts a twisted
higher dimensional analog
of the Milnor Wood inequality.

The goal of this article is to provide a geometric perspective on
the topology of surface bundles over a surface. 
We begin with summarizing some 
constructions
of Kodaira fibrations
in Section \ref{constructions}.
In Section \ref{flatcircle} we give a
geometric proof of
Morita's theorem (see also Chapter 5 of \cite{FM12}). 

A \emph{section} of a surface 
bundle $\Pi:E\to B$ is a smooth map $\sigma:B\to E$ such that
$\Pi\circ \sigma={\rm Id}$.
In Section \ref{selfintersection}
we study the self-intersection number of a section 
of a surface bundle over a surface and compute all such 
self-intersection numbers for the trivial bundle. We also
point out that Morita's theorem yields 
an elementary and purely topological 
proof of the following extension
of Proposition 1 of \cite{BKM13} (see also \cite{Bow11}) 
which was originally established using Seiberg Witten invariants.

\begin{prop}\label{evaluation}
Let $E\to B$ be a surface bundle over a surface.
Let $\Sigma$ be closed
surface and let $f:\Sigma\to E$ be a smooth map; then 
\[\vert c_1(f^*\nu)(\Sigma)\vert \leq \vert \chi(\Sigma)\vert.\]
\end{prop}

In Section \ref{cohomology} we describe 
the second integral cohomology group of a surface bundle over a smooth
base in an explicit
way as the group of first Chern classes of complex line bundles.
We apply this discussion to show an analog of a result of
Morita who computed the cohomology of a surface bundle with
rational coefficients 
(Proposition 3.1 of \cite{Mo87}). The following is also related to the work 
of Harer \cite{H83} who computed for $g\geq 5$ the second homology group of 
${\cal M}_g$ with integral coefficients (see also the more recent and more
complete account in \cite{KS03}). 

\begin{theo}
Let $E\to B$ be a surface bundle with fiber genus $g$. If $E$ admits a section then
there exists an integral class $e\in H^2(E,\mathbb{Z})$ such that
$(2g-2)e=c_1(\nu)$, and 
\[H^p(E,\mathbb{Z})\equiv H^p(B,\mathbb{Z}) \oplus 
H^{p-1}(B;H^1(S_g,\mathbb{Z})) \oplus eH^{p-2}(B,\mathbb{Z})\]
\end{theo}

In view of the result of Chen and Salter \cite{CS18} that the pullback of the 
universal curve to any finite orbifold cover of ${\cal M}_g$ does not admit
a section, it seems that most surface bundles do not admit sections. 

\bigskip
\noindent 
{\bf Acknowledgement:} The completion of this work was 
supported by the National Science Foundation under Grant No. 1440140, while the author was in residence at the Mathematical Sciences Research Institute in Berkeley, California, during the fall semester of 2019.

\section{Constructions}\label{constructions}

In this section we review some constructions of Kodaira fibrations 
from the literature.

The best known construction method of Kodaira fibrations
goes back to an idea of  
Atiyah and Kodaira. Their examples are 
 branched covers over
a product of two complex curves, and they fiber in two different ways.  
A variation of this idea was used by 
Bryan and Donagi \cite{BD02} to show that 
for any integers $h,n\geq 2$, there exists a connected algebraic surface
$X_{h,n}$ of signature $\sigma(X_{h,n})=\frac{4}{3}h(h-1)(n^2-1)n^{2h-3}$ that 
admits two smooth fibrations $\theta_1:X_{h,n}\to C$ and 
$\theta_2:X_{h,n}\to D$ with base and fiber genus
$(b_i,f_i)$ equal to 
\begin{align}\label{basefiber}
(b_1,f_1)&= (h,h(hn-1)n^{2h-2}+1)\text{ and }\notag\\
(b_2,f_2)&=(h(h-1)n^{2h-2}+1,hn) \notag
\end{align}
respectively.
Note that the smallest
fibre genus of the surfaces in the above family equals 4.

Taking $n=h=2$, we conclude that there is a surface bundle
with fiber genus $4$ and base genus $9$ with 
\[\frac{\sigma(E)}{\chi(E)}=\frac{16}{96}=\frac{1}{6}.\]
According to my knowledge, 
this is the example with the largest known ratio between
signature and Euler characteristic. 

Complete intersections provide 
a more indirect way to construct Kodaira fibrations 
(see \cite{Ar17} for a recent discussion). 
To be more specific, call a Kodaira fibration
with fiber $S_g$ of genus $g$ \emph{generic}
if the fundamental group $\pi_1(B)$ of the base surjects onto 
a finite index subgroup of the mapping class group 
$\Gamma_g$ by the monodromy 
homomorphism $\rho$. 
Generic Kodaira fibrations
can be constructed as follows. 

The action of a diffeomorphism of $S_g$ on the first homology group
$H_1(S_g,\mathbb{Z})$ preserves the intersection form 
and only depends on the isotopy class of the diffeomorphism. Thus there exists
a surjective \cite{FM12} homomorphism 
\[\Psi:\Gamma_g\to {\rm Sp}(2g,\mathbb{Z}).\] 
For every $n\geq 3$, the kernel
of the induced homomorphism $\Gamma_g\to {\rm Sp}(2g,\mathbb{Z}/n\mathbb{Z})$
is torsion free and determines the \emph{fine moduli space of genus $g$ curves
  with level $n$ structure} ${\cal M}_{g}[n]$, which is a complex manifold.

Let ${\cal M}_{g}[n]^*$ be the Satake compactification of 
${\cal M}_{g}[n]$. If $g\geq 3$, then the  boundary
${\cal M}_{g}[n]^*-{\cal M}_{g}[n]$ has complex codimension at least 2.  
Therefore a curve $C\subset {\cal M}_{g}[n]^*$ given as an intersection of 
general ample divisors lies entirely in ${\cal M}_{g}[n]$. By the weak
Lefschetz theorem, the inclusion $C\to {\cal M}_{g}[n]$ induces
a surjection of fundamental groups. As the fundamental group of
${\cal M}_{g}[n]$ is a finite index subgroup of the mapping class group, 
the restriction to $C$ 
 of the universal curve defines a generic Kodaira fibration.

The Atiyah Kodaira  
examples which are branched covers over the product of two complex curves
are not generic. Namely, if $E\to B$ is a generic Kodaira fibration, then the image
of the monodromy group under the homomorphism $\Psi$ 
is a Zariski dense subgroup of $Sp(2g,\mathbb{R})$. However, if $E\to B$ is  
an Atiyah Kodaira example, then the
group $\Psi(\rho(\pi_1(B)))$ fixes a symplectic plane in 
$H^1(B,\mathbb{R})$ and hence by duality, $\Psi(\rho(\pi_1(B)))$ is not Zariski dense in 
$Sp(2g,\mathbb{R})$. 
Interestingly, Flapan \cite{Fl17} 
used complete intersections to construct Kodaira fibrations with fiber of genus 3
which also have this property. She also classified all $\mathbb{Q}$-algebraic subgroups of 
$Sp(6,\mathbb{R})$ which arise as the smallest algebraic group containing the 
image of the monodromy group of a Kodaira fibration with fiber of genus $3$.

On the other hand,
Bregman \cite{Br18} established that variations of 
the Atiyah Kodaira construction may
in some sense be universal for Kodaira fibrations whose monodromy 
fixes a symplectic plane in $H^1(B,\mathbb{R})$. 
He showed that if the dimension of the fixed point set
of the mondromy of a Kodaira fibration $E$
acting on the holomorphic one-forms of a fixed fiber equals $d$ for some
$1\leq d\leq 2$, then 
there exists a genus $d$ curve $D$ and a ramified covering $F:E\to D\times B$
inducing an isomorphism on first cohomology with rational coefficients. 

There are also explicit constructions of surface bundles over surfaces with
nontrivial signature which do not admit a complex structure
\cite{B11}. The fiber genus of such a surface bundle is
at least 4. The article \cite{EKS02} constructs surface bundles over
surfaces with positive signature for any fiber genus $g\geq 3$.
In contrast, the signature
of a surface bundle over a surface with fiber genus 2 always vanishes.
This follows from the fact that the second cohomology group 
$H^2(\Gamma_g,\mathbb{Z})$ is isomorphic to 
$H_2(\Gamma_g,\mathbb{Z})/{\rm torsion}$ (see p.158 of \cite{FM12}), on the
other hand we have $H_2(\Gamma_2,\mathbb{Z})=\mathbb{Z}/2\mathbb{Z}$
\cite{KS03}.

\section{Flat circle bundles}\label{flatcircle}

Consider the universal curve
$\Pi:{\cal U}\to {\cal M}_g$  
over the moduli space ${\cal M}_g$  
of genus $g$ curves. Its fiber over a 
point $X\in {\cal M}_g$ is just the complex curve 
$X$. The tangent bundle $\nu$ of the fibers of this bundle
is a holomorphic complex line bundle on the complex orbifold
${\cal U}$.

The following observation (which is due to Morita \cite{Mo88})
is based on some facts which 
were probably already known to Nielsen. 

\begin{proposition}\label{flatagain}
The circle subbundle of the bundle $\nu\to {\cal U}$ is flat.
\end{proposition}
\begin{proof} Let $\Gamma_{g,1}$ be the mapping class group of a surface of 
genus $g$ with one marked point (puncture), and denote by 
$\Theta:\Gamma_{g,1}\to \Gamma_g$ the homomorphism induced by the puncture forgetful map.
This homomorphism fits into the 
\emph{Birman exact sequence} \cite{B74,FM12}
\[1\to \pi_1(S_g)\to \Gamma_{g,1}\xrightarrow{\Theta} \Gamma_g\to 1.\]
Via this sequence, the group $\Gamma_{g,1}$ is the orbifold fundamental group of the 
universal curve.

We claim that the group $\Gamma_{g,1}$ 
admits an action on the circle $S^1$ by orientation preserving
homeomorphisms, where
we view $S^1$ as the ideal boundary $\partial {\mathbb{H}^2}$ of the hyperbolic plane
$\mathbb{H}^2$. 
Namely, let $x\in S_g$ be a fixed point. 
The group $\Gamma_{g,1}$ can be viewed as the group of 
isotopy classes of orientation preserving
diffeomorphisms of the surface $S_g$ preserving
$x$. Isotopies are also required to fix $x$. 
Any orientation preserving 
diffeomorphism $f$ of $S_g$ which fixes $x$ 
induces an automorphism $f_*\in {\rm Aut}(\pi_1(S_g,x))$, and
since the group of diffeomorphisms of $S_g$ isotopic to the identity
is contractible, the isotopy class of $f$ is uniquely determined by the induced
map $f_*$. 

The group $PSL(2,\mathbb{R})$ is just the group of 
orientation preserving isometries of the hyperbolic plane $\mathbb{H}^2$, or, 
equivalently, it is the group of biholomorphic automorphisms of the
unit disk $D\subset \mathbb{C}$. 
The choice of a hyperbolic structure on $S_g$ then 
determines the conjugacy class of an
embedding $\pi_1(S_g)\to PSL(2,\mathbb{R})$,
with discrete cocompact image.

Since the group $PSL(2,\mathbb{R})$ acts simply transitively
on the unit tangent bundle $T^1\mathbb{H}^2$ of the hyperbolic plane,
we can choose an identification of $PSL(2,\mathbb{R})$ with
$T^1\mathbb{H}^2$ which maps the identity to the point 
$0\in D=\mathbb{H}^2$. We also may assume that $0$ is a preimage of 
the point $x\in S_g$. With these identifications, the
group $\pi_1(S_g,x)$ determines an embedding 
$\pi_1(S_g,x)\to PSL(2,\mathbb{R})$, unique up to conjugation with
the central subgroup $SO(2)\subset PSL(2,\mathbb{R})$, that is, the stabilizer of 
the basepoint $0$. Fix once and for
all such an embedding.

A diffeomorphism $f$ of $S_g$ fixing $x$ is a 
bilipschitz map for the
hyperbolic structure of $S_g$. 
Thus $f$ can be lifted to a 
$\pi_1(S_g,x)$-equivariant bilipschitz map 
$\tilde f:\mathbb{H}^2=D\to D$ which fixes $0$. This means that
the map $\tilde f$ satisfies $\tilde f(\psi y)=f_*(\psi)(\tilde f(y))$ for all 
$y\in D$ and all $\psi\in \pi_1(S_g,x)\subset PSL(2,\mathbb{R})$.
Equivariance and the requirement that $\tilde f(0)=0$ determines
the lift $\tilde f$ completely.

Now any bilipschitz map of the hyperbolic plane which fixes the point $0$ 
maps geodesic
rays beginning at $0$ 
to uniform quasi-geodesic rays issuing from the same point.  
Such a uniform quasi-geodesic ray is at uniformly bounded distance from a 
geodesic ray, and this geodesic ray is unique if its starting point  
is required to be the fixed point $0$. 
As the ideal boundary $\partial D=S^1$ 
of the hyperbolic plane is just the set
of geodesic rays issuing from $0$, this shows that the map
$\tilde f$ induces a 
homeomorphism $\Upsilon(f)\in {\rm Top}^+(S^1)$. Here preservation 
of orientation of $\Upsilon(f)$ follows from preservation of orientation of $f$. 
The homeomorphism $\Upsilon(f)$  
only depends on the isotopy class
of $f$ provided that such an isotopy fixes the point $x$.

By construction, if $u$ is another 
orientation preserving bilipschitz
homemorphism of $S_g$ fixing $x$, then 
$\Upsilon(u\circ f)=\Upsilon(u)\circ \Upsilon(f)$. As a consequence,
the assignment $f\to \Upsilon(f)$ which associates to a diffeomorphism
$f$ of $S_g$ fixing $x$ the homeomorphism $\Upsilon(f)\in {\rm Top}^+(S^1)$ 
descends to a homomorphism $\hat \Upsilon:\Gamma_{g,1}\to {\rm Top}^+(S^1)$,
unique up to conjugation in ${\rm Top}^+(S^1)$. 
This homomorphism then defines a flat circle bundle 
$H\to {\cal U}$. 
We claim that this circle bundle is (up to equivalence)
the circle subbundle of the
vertical tangent bundle of the surface bundle ${\cal U}$.

Namely, consider the space ${\cal T}(S_{g}^1)$
 of all discrete faithful orientation preserving
homomorphisms $\rho:\pi_1(S_g,x)\to PSL(2,\mathbb{R})$. 
The group $PSL(2,\mathbb{R})$ acts on this space by conjugation.
Each orbit of this action 
is a fiber of the bundle ${\cal T}(S_{g}^1)\to {\cal T}(S_g)$ where ${\cal T}(S_g)$ denotes
the Teichm\"uller space of marked complex structures on the closed oriented
surface $S_g$ of genus $g$. 
The quotient of ${\cal T}(S_g^1)$ by the action of 
the central circle group $SO(2)=S^1$ is the Teichm\"uller space 
${\cal T}(S_{g,1})$ of all marked complex structures on an oriented surface 
$S_{g,1}$ of genus $g$ with one marked point.

The circle bundle $PSL(2,\mathbb{R})=T^1\mathbb{H}^2\to \mathbb{H}^2$ 
has a $PSL(2,\mathbb{R})$-equivariant
identification with
$\mathbb{H}^2\times S^1$ where the action of $PSL(2,\mathbb{R})$
on $S^1=\partial \mathbb{H}^2$ is described as follows. 

For a unit tangent vector $u\in T^1\mathbb{H}^2$
let $\gamma_u$ be the geodesic ray with initial velocity $u$.
The projection of the identity in $PSL(2,\mathbb{R})$ is a fixed basepoint 
$0\in D=\mathbb{H}^2$. Identify $\partial \mathbb{H}^2$ with the fiber of the unit tangent
bundle over this basepoint by associating to a unit tangent vector 
$v\in T_0{\mathbb{H}}^2$ the endpoint $\gamma_v(\infty)$ 
of the geodesic ray $\gamma_v$.
For each $\alpha\in PSL(2,\mathbb{R})$, the differential 
$d\alpha(0)$ of $\alpha$ at the basepoint $0$ then induces the homeomorphism
$\gamma_v(\infty)\to \gamma_{d\alpha(v)}(\infty)$
of $\partial \mathbb{H}^2=S^1$. The thus defined map
$PSL(2,\mathbb{R})\to {\rm Top}^+(S^1)$ is 
equivariant with respect to the action of $PSL(2,\mathbb{R})$
on itself by conjugation.

Now let $f:(S_g,x)\to (S_g,x)$ be an arbitrary diffeomorphism which fixes the point $x$.
Taking a quotient by the action of $(0,\infty)$ on the tangent bundle of $S_g$ by scaling
shows that its differential induces an
isomorphism of the circle bundle $T^1S_g$ covering the
base map $f$. 
On the other hand, using the above construction,
the map $f$ induces a second isomorphism of $T^1S_g$ as follows.
Lift $f$ to a diffeomorphism $\tilde f$ of $\mathbb{H}^2$ which fixes
$0$ and is equivariant with respect to the action of
$\pi_1(S_g,x)$ and its image under the automorphism 
$f_*$. For any point $y\in \mathbb{H}^2$, map
a unit tangent vector $v\in T_y^1\mathbb{H}^2$ to the
unit tangent vector $w\in T_{\tilde f(y)}\mathbb{H}^2$ such that
$\gamma_w(\infty)=\Upsilon(f)\gamma_v(\infty)$.
As this construction depends continuously on $v$ and is equivariant
with respect to the action of $\pi_1(S_g,x)$, it descends to
an isomorphism of $T^1S_g$ covering $f$.

The circle subbundle of the vertical tangent bundle of the universal curve 
${\cal U}$ is the quotient of the vertical tangent bundle of 
${\cal T}(S_{g,1})$, that is of ${\cal T}(S_g^1)$, 
by the action of $\Gamma_{g,1}$ via the tangent map of isotopy classes of
diffeomorphisms 
fixing the basepoint $x$. Thus to show that this circle bundle
is indeed the flat bundle 
defined by the homomorphism $\hat\Upsilon:\Gamma_{g,1}\to {\rm Top}^+(S^1)$, 
it suffices to show that for any diffeomorphism $f:(S_g,x)\to (S_g,x)$,
the isomorphism of $T^1S_g$ induced by $df$ is homotopic to the isomorphism
induced by $\Upsilon(f)$.

To show that this is indeed the case
lift as before the diffeomorphism $f$ to an $f_*$-equivariant diffeomorphism $\tilde f$ of $\mathbb{H}^2$ 
fixing $0$. We deform equivariantly the tangent map $d\tilde f$ of $\tilde f$ as follows.

For a number $r>0$ and a point $y\in {\mathbb H}^2$, identify the fiber of the unit tangent bundle 
of ${\mathbb H}^2$ at $y$ with the boundary $\partial B(y,r)$ 
of the ball of radius $r$ about $y$ using the exponential
map of the hyperbolic plane. The image $\tilde f(\partial B(y,r))$ bounds a disk containing $\tilde f(y)$.
Use the inverse of the exponential map at $\tilde f(y)$ to map this circle onto the fiber of the 
unit tangent bundle of ${\mathbb H}^2$ at $\tilde f(y)$.
Doing this simultaneously for all $y\in {\mathbb H}^2$ 
defines a continuous map
$\tilde \zeta_r:T^1{\mathbb H}^2\to T^1{\mathbb H}^2$
which is equivariant with respect to 
the action of $\pi_1(S_g)$
and hence descends to a continuous map 
$\zeta_r:T^1S_g\to T^1S_g$ covering $f$. 
Clearly the maps $\zeta_r$ 
depend continuously on $r$, and as $r\to 0$, they converge to the map
induced by $df$.
Thus for all $r$, the map $\zeta_r$ is homotopic to the map induced by $df$. 

Now by construction, as $r\to \infty$ the maps $\zeta_r$ converge to the map induced by $\Upsilon(f)$.
As this construction is moreover equivariant with respect to isotopy, 
this shows that the action of $\Gamma_{g,1}$ on the vertical tangent bundle of the universal covering
of the universal curve ${\cal U}$ coincides with the action 
defined by the homomorphism $\hat \Upsilon$. 
Hence the flat circle bundle on ${\cal U}$ 
constructed from the homomorphism $\hat\Upsilon:\Gamma_{g,1}\to {\rm Top}^+(S^1)$ indeed equals
the circle subbundle of the vertical tangent bundle of ${\cal U}$. 
This is what we wanted to show.
\end{proof}

Let now $\Pi:E\to B$ be a surface bundle
over an arbitrary smooth base $B$, with fibre 
$S_g$ of genus $g\geq 2$. Any such surface bundle can
be obtained as a pull-back of the universal curve ${\cal U}\to {\cal M}_g$ 
by a smooth (in the sense of orbifolds) map 
$f:B\to {\cal M}_g$. Thus we may assume that 
the fibres of $E$ are equipped with a complex structure
varying smoothly over the base. As a consequence,
the vertical tangent bundle $\nu$ of $E$
is a smooth complex line bundle over $E$.

Since ${\cal M}_g$ is a classifying space 
(in the orbifold sense) for its (orbifold) fundamental group,
the homotopy class of a map $f:B\to {\cal M}_g$ is uniquely determined
by the induced homomorphism $f_*=\rho:\pi_1(B)\to \Gamma_g$. 
Here as before, $\Gamma_g$ is the mapping class group of $S_g$.
Furthermore, homotopic maps give rise to homeomorphic surface
bundles, so the bundle $E$ is determined by $\rho$.
We refer to \cite{Mo87} for more
details about these well known facts. 

Let as before $\Theta:\Gamma_{g,1}\to \Gamma_g$ be the natural surjective homomorphism. 
Since $E$ is the pull-back of the universal curve under the map $f$, 
there exists an
exact diagram
\begin{equation}
\begin{tikzcd}
1\arrow[r]&\pi_1(S_g) 
\arrow[r] \arrow[d] &\pi_1(E)\arrow[r]\arrow[d]&\pi_1(B)
\arrow[r]\arrow[d] & 1\\
1\arrow[r] &\pi_1(S_g) 
\arrow[r]  &\Theta^{-1}(\rho(\pi_1(B))) 
\arrow[r] & \rho(\pi_1(B))\arrow[r] &1.
\end{tikzcd} 
\end{equation}
As a consequence, there exists a homomorphism
$\pi_1(E)\to \Theta^{-1}(\rho(\pi_1(B)))\subset \Gamma_{g,1}$ whose restriction
to the subgroup $\pi_1(S_g)$ is an isomorphism. 
By naturality under pull-back, in the case that $B$ is a closed surface we obtain

\begin{corollary}\label{tangentbundle}
Let $\Pi:E\to B$ be a surface bundle over a surface. Then $TE=\nu\oplus 
\Pi^*TB$ is a sum of two complex line bundles whose circle subbundles 
are flat.
\end{corollary}
\begin{proof}
We observed before that $TE=\nu\oplus \Pi^*TB$, and by
Proposition \ref{flatagain}, the circle subbundle of $\nu$ is a pull-back of a flat bundle
and hence flat. 
On the other hand, as $B$ is a surface of genus $h\geq 2$, the
circle subbundle of the tangent bundle $TB$ of $B$ is flat as well and
hence the same holds true for the circle subbundle of the pull-back
$\Pi^*TB$. 
\end{proof}

\section{Selfintersection numbers of sections}
\label{selfintersection}

A \emph{section} of a surface bundle
$\Pi:E\to B$ is a smooth map $f:B\to E$ so that
$\Pi\circ f={\rm Id}$. 
The image $f(B)$ of a section
$f$ is a cycle in $E$ which defines a homology class
$[f(B)]\in H_k(E,\mathbb{Z})$ where $k={\rm dim}(B)$. 
In the case that $B$ is a surface, the self-intersection number
$[f(B)]^2$ of this class is defined. 
Our next goal is to shed some light on 
this self-intersection number
from a geometric point of view.

Let as before $\nu$ be the vertical line bundle of $E$, with 
first Chern class $c_1(\nu)$.
Equivalently,
$c_1(\nu)$ is the Euler class of the oriented 2-dimensional real oriented
vector bundle $\nu$. 
We note

\begin{lemma}\label{eulernumber}
$[f(B)]^2=c_1(\nu)(f(B))$ for any section $f:B\to E$.
\end{lemma}
\begin{proof} Since $f(B)$ is a smoothly  embedded surface in $E$,
the self-intersection number of $f(B)$ equals the
Euler number of the pull-back to $B$ of the 
oriented normal bundle of $f(B)$ in $E$, that is, it equals the
evaluation of the Euler class of this normal bundle on the homology class
$[f(B)]$. 

As $f$ is a section, $f(B)$ is everywhere transverse to the fibers of 
$E\to B$. Thus this oriented normal bundle is 
isomorphic to the restriction of the vertical tangent bundle 
$\nu$ of $E$.
\end{proof}

It was shown by 
Milnor \cite{M58} 
and Wood \cite{W71} that the Euler number $e(H)$ 
of a flat circle bundle $H\to B$ over a closed oriented
surface $B$ of genus $h\geq 2$ 
and the Euler characteristic 
$\chi(B)$ of $B$ satisfy the inequality 
\[\vert e(H)\vert \leq \vert \chi(B)\vert.\]
In view of this result, the conjecture
stated in the introduction can 
be viewed as a higher dimensional
analog of the Milnor Wood inequality.

By a result of Thom, 
for any compact $CW$-complex $X$, any homology class
$\alpha\in H_2(X,\mathbb{Z})$ can be represented by a map
from a closed surface into $X$, and if 
$\alpha$ is not a two-torsion class, then the surface can be
chosen to be orientable.   
As a consequence of Proposition \ref{flatagain}, we obtain

\begin{corollary}\label{milnorwood}
Let $\beta\in H_2(E,\mathbb{Z})$ be 
represented by a map $f:\Sigma\to E$ where $\Sigma$ is a closed
oriented surface. 
Then $\vert c_1(\nu)(\beta)\vert \leq \vert \chi(\Sigma)\vert$.
\end{corollary}
\begin{proof} By Proposition \ref{flatagain}, 
the pull-back by $f$ of the circle subbundle of $\nu$ is a flat circle
bundle over $\Sigma$. By naturality, we have
\[\vert c_1(\nu)(\beta)\vert =\vert f^*(c_1(\nu))(\Sigma)\vert \leq 
\vert \chi(\Sigma)\vert \]
by the Milnor Wood inequality. 
\end{proof}

As an immediate consequence, we obtain 
the following result 
of Baykur, Korkmaz and Monden 
(Proposition 1 of \cite{BKM13}) and Bowden \cite{Bow11}, bypassing the 
use of Seiberg-Witten invariants used 
to derive this statement in \cite{BKM13} and \cite{Bow11}. 

\begin{corollary}\label{selfinter}
Let $f:B\to E$ be a section of a surface bundle
$E\to B$; then 
$\vert [f(B)]^2\vert \leq \vert \chi(B)\vert$. 
\end{corollary}
\begin{proof} 
The section is defined by a smooth map 
$B\to E$ and hence the corollary follows from Lemma \ref{eulernumber} and
Corollary \ref{milnorwood}.
\end{proof}

%
%
%

Theorem 15 of \cite{BKM13} shows that 
for every $g\geq 2,h\geq 1$ and every integer $k\in [-2h+2,2h-2]$ 
there is
a surface bundle with fibre $S_g$ and base of genus $h$ 
which admits 
a section of self-intersection number $k$.
We complement this result by analyzing self-intersection numbers
of sections of the trivial bundle.

\begin{proposition}\label{trivialbundle}
Let $E\to B$ be the trivial surface bundle with fibre 
genus $g\geq 2$ and base genus $h$.
If $h<g$ then every section of $E$ has self-intersection number
zero. If $h\geq g$ then for each integer $k$ with
$h-1\geq \vert k\vert (g-1)$ 
there is a section of 
self-intersection number 
$2k(g-1)$, and no other self-intersection numbers occur.
\end{proposition}
\begin{proof}
Let $B$ be a surface of genus $h\geq 1$
and let $E=B\times S_g\to B$ be the trivial surface bundle.
Then a section $f:B\to E$ is just a smooth 
map of the form $x\to (x,\Phi(x))$ where 
$\Phi:B\to S_g$ is smooth. 
Let $d\in \mathbb{Z}$ be
the degree of $\Phi$. We claim that the self-intersection
number of $f$ equals $d(2-2g)$.

To see that this is the case, denoting as before by $\nu$
the vertical tangent bundle, we have 
$c_1(\nu)(f(B))=\Phi^*c_1(TS_g)(B)=d(2-2g)$. 
By Lemma \ref{eulernumber},
the self-intersection number of the section $f$ 
coincides with this Euler number.

This shows that the self-intersection number of a 
section of the trivial
bundle is a multiple of $2g-2$. The proposition now
follows from the fact that for a surface $B$ of genus $h$ and every 
$k\in \mathbb{Z}$, there exists
a smooth map $\Phi:B\to S_g$ of degree $k$ if and only if 
$\vert k\vert \leq \frac{h-1}{g-1}$.
Then $x\in B\to (x,\Phi(x))$ is a section of 
$E\to B$ with self-intersection number $k(2-2g)$. 

To show that the condition on $k$ is sufficient for the existence 
of a map $B\to S_g$ of degree $k$, note that
if $k\leq \frac{h-1}{g-1}$ is positive, then a map $ B\to S_g$ of 
degree $k$ can be 
constructed as follows. Let $\psi:\Sigma\to S_g$ be an unbranched
cover of degree $k$. The Euler characteristic of $\Sigma$ fulfills 
$\vert \chi(\Sigma)\vert =k\vert \chi(S_g)\vert 
\leq \vert \chi(B)\vert$.
Thus there is a degree one map 
$\zeta:B\to \Sigma$ which pinches
a subsurface of $B$ of genus $g^\prime$ to a point,
where $g^\prime=h-1-k(g-1)\geq 0$ \cite{Ed79}.  
The composition
$\psi\circ \zeta:B\to S_g$ has degree $k$. A map of degree $-k$
can be taken as a composition of this map 
with an orientation reversing diffeomorphism of $B$.

On the other hand, by a result of Edwards \cite{Ed79}, any
map $B\to S_g$ of degree $k\geq 10$ is homotopic to the composition of
a pinch $B\to B^\prime$ with a branched cover $B^\prime\to S_g$.
A nontrivial pinch collapses a subsurface of $B$ bounded by an
essential separating simple closed curve to a point and hence it 
strictly decreases the genus. Thus the genus $q$ of $B^\prime$ is not bigger
than the genus $h$ of $B$. 

Now if 
$B^\prime\to S_g$ is a branched cover of degree $k$ and if  
$b$ is the total number of branch points, counted with multiplicity, then by the
Hurwitz formula \cite{FK80}, 
\[2q-2=b+k(2g-2).\]
This implies that $\vert k\vert \leq \frac{q-1}{g-1}$ and hence
$\vert k\vert \leq \frac{h-1}{g-1}$.  As any map $B\to S_g$ can be precomposed
with an orientation reversing diffeomorphism of $B$ to yield a map of positive degree, 
this completes the proof of the proposition.
\end{proof}



\section{Cohomology of surface bundles}\label{cohomology}

This final section collects some results on the second cohomology group 
of a surface bundle over a base $B$
which is an arbitrary smooth closed manifold. We also give some additional
information in the case that $B$ is a surface.

The cohomology with rational coefficients of a surface bundle over a smooth base
was computed by Morita. The following is 
Proposition 3.1 of \cite{Mo87}.

\begin{proposition}\label{cohomologyprop}
  Let $\Pi:E\to B$ be a surface bundle over a smooth base $B$.
  Let $k=\mathbb{Q}$ or $\mathbb{Z}/p\mathbb{Z}$ where
  $p$ is a prime not dividing $2g-2$. 
  Then the homomorphism
$\Pi^*:H^*(B,k)\to H^*(E,k)$ is injective, and for all $q$ we have 
\[H^q(E,k)\cong H^q(B,k)
  \oplus H^{q-1}(B,H^1(S_g,k))\oplus
c_1(\nu)H^{q-2}(B,k).\]
\end{proposition}

In general, we can not hope that the proposition passes on 
to cohomology with integral coefficients. The reason is that 
for a surface bundle $E\to B$ with fiber of genus $g\geq 2$, 
the fiber inclusion $\iota:S_g\to E$ may not induce a 
surjection $\iota^*:H^2(E,\mathbb{Z})\to H^2(S_g,\mathbb{Z})=\mathbb{Z}$.
Namely, the Euler class $e\in H^2(S_g,\mathbb{Z})$ of the tangent bundle of 
$S_g$ has a $2g-2$-th root, but there may not exist such a root 
for the class $c_1(\nu)$ where as before, $c_1(\nu)\in H^2(E,\mathbb{Z})$ 
denotes the first Chern class of the
vertical tangent bundle of $E$.

In the following observation, the component $H^1(B,H^1(S_g,\mathbb{Q}))\subset 
H^2(E,\mathbb{Q})$ is as in Proposition \ref{cohomologyprop}.

\begin{proposition}\label{surfacebundles}
Let $E\to B$ be a surface bundle over a surface. Then there exists an embedding
$H^1(B,H^1(S_g,\mathbb{Z}))\to H^2(E,\mathbb{Z})$ which induces an isomorphism
$H^1(B,H^1(S_g,\mathbb{Z}))\otimes \mathbb{Q}\to H^1(B,H^1(S_g,\mathbb{Q}))\subset 
H^2(E,\mathbb{Q})$.
\end{proposition}
\begin{proof} 
The standard Leray spectral sequence for the fiber bundle 
$\Pi:E\to B$ starts with a finite good cover ${\cal U}=\{U_i\mid 1\leq i\leq k\}$ 
of $B$ consisting of open sets 
$U_i\in {\cal U}$ with the following properties.
\begin{enumerate}
\item Each set $U_i\in {\cal U}$ is diffeomorphic to an open disk $D\subset \mathbb{R}^2$.
\item Each intersection $U_i\cap U_j$ or $U_i\cap U_j\cap U_k$ 
is contractible or empty.
\item The intersection of any four distinct of the sets $U_i$ is empty.
\end{enumerate}
Such a covering can be constructed from a triangulation $T$ of $B$ as follows.

For each vertex $x$ of $T$, choose a disk neighborhood $D_x$ of $x$ so that the closures of these disks
are pairwise disjoint. We also require that each edge $e$ of $T$ intersects a disk $D_x$ if and only if
the edge is incident on $x$, and in this case, the intersection of $e$ with $D_x$ is a connected 
subarc of $e$. Furthermore, we require that a two-simplex $f$ intersects a disk $D_x$ if and only if 
$x$ is a vertex of $f$, and In this case, $D_x\cap f$ is a disk. Call these disks of vertex type. 

For each edge $e$ of $T$, choose a disk $D_e$ containing a neighborhood of $e-\cup_xD_x$. 
This can be done in such a way that the disks $D_e$ are pairwise disjoint, that 
for a vertex $x$ of $T$, the intersection $D_e\cap D_x$ is empty or a disk, and that the later holds true if and 
only if $e$ is incident on $x$. Call such a disk of edge type. The union of the disks of vertex and edge type 
covers a neighborhood of the 
one-skeleton of $T$, and the intersection of any two of these disks either is a disk or empty. 
The intersection of any three of the disks is empty. 

Finally choose a disk for each two-simplex $f$ of $T$ which is contained in $f$ and covers 
$f-\cup_eD_e-\cup_xD_x$. Call these disks of face type. They 
can be chosen in such a way that the resulting family of disks
covers $B$ and that furthermore, if the intersection of any three of the disks is non-empty, then 
these disks are of distinct type. The resulting cover ${\cal U}$ 
is called a \emph{good cover} of $B$.
The restriction of $E$ to $U_i$ is trivial for all $U_i\in {\cal U}$.

Th good cover ${\cal U}$ determines a first quadrant  double chain complex $K^{p,q}$ 
with $0\leq p\leq 2$ which can be used to compute $H^2(E,\mathbb{Z})$ using the Leray spectral
sequence for the sheaf ${\cal F}$ of locally constant $\mathbb{Z}$-valued functions on $E$.
Leray's theorem 
(see Theorem 14.18 of \cite{BT82}) shows that 
the $E_2$-term of the spectral sequence 
with coefficients $\mathbb{Z}$ has the form 
\[E_2^{p,q}=H^p({\cal U},H^q(S_g,\mathbb{Z})).\]
This spectral sequence converges to $H^*(E,\mathbb{Z})$ by the generalized Mayer-Vietoris
principle, because $\Pi^{-1}({\cal U})$ is a cover of $E$, see p.169 and Theorem 15.11 of \cite{BT82} for 
detials on these facts.

Since $K^{p,q}$ is trivial for $p\geq 3$,
for $r\geq 2$ and every $k\geq 2$ the differential
$d_r:E_k^{1,1}\to E_k^{1+r,2-r}$ vanishes. 
Since the spectral sequence converges to $H^*(E,\mathbb{Z})$, this implies that indeed we have
an embedding $H^1({\cal U},H^1(S_g,\mathbb{Z})=
H^1(B,H^1(S_g,\mathbb{Z}))\to H^2(E,\mathbb{Z})$. 

We claim that the image group is precisely the subgroup of $H^2(E,\mathbb{Z})$ whose 
cup product with $C=\Pi^*H^2(B,\mathbb{Z})\oplus c_1(\nu)H^0(B,\mathbb{Z})$ vanishes.
Note to this end that $C$ is a free subgroup of $H^2(E,\mathbb{Z})$ of rank two, and the restriction of the cup product
to $C$ is non-degenerate. Now the degree two part of the $E_2$-term of the spectral sequence decomposes
as $E_2=E_2^{2,0}\oplus E_2^{1,1}\oplus E_2^{0,2}$. The cup product defines a homomorphism
\[E_2^{2,0}\otimes E_2^{1,1}\to E_2^{3,1}=0,\] and similarly, the cup product defines a homomorphism
\[E_2^{1,1}\otimes E_2^{0,2}\to E_2^{1,3}=0.\]

As this argument is also valid with coefficients in $\mathbb{Q}$, and cup product is natural with respect to taking
tensor product with $\mathbb{Q}$, this completes the proof. 
\end{proof}

\begin{remark}
We do not know an example of a surface bundle for which the conclusion of
Proposition \ref{surfacebundles} is violated. In particular, by \cite{H83}, it holds true
for the universal curve, in fact we have
$H^1(\Gamma_g,H^1(S_g,\mathbb{Z}))=0$ for all $g$.
\end{remark}

%
%

The final goal of this article is 
to give a geometric interpretation of 
the subgroup
$H^1(B,H^1(S_g,\mathbb{Z}))$ of $H^2(E,\mathbb{Z})$ for a surface bundle
over a surface 
$\Pi:E\to B$ and prove the theorem from the introduction.
We begin with some results which hold true for an arbitrary surface bundle 
$E\to B$ over a smooth base. 
Assume as before that $E\to B$ is obtained by a
smooth map $B\to {\cal M}_g$. This means that each of the fibers of
$E$ has a complex structure varying smoothly over the base.

Abel's theorem shows that the \emph{Picard group} ${\rm Pic}(X,2g-2)$ 
of all complex line bundles 
of degree $2g-2$ over a Riemann surface $X$ can be identified with the
\emph{Jacobian} ${\cal J}(X)$ of $X$ as follows \cite{FK80,GH78}.

Choose a geometric symplectic basis $a_1,b_1,\dots,a_g,b_g$ of 
$H_1(S_g,\mathbb{Z})$. This means that
$a_i,b_i$ are oriented non-separating simple closed curves in 
$S_g$ so that $a_i,b_i$ intersect in precisely one point, and 
$a_i\cap a_j=a_i\cap b_j=b_i\cap b_j=\emptyset$ for all $i\not=j$. 
This choice then determines a basis $\omega_1,\dots,\omega_g$ of 
the $g$-dimensional \cite{FK80} complex 
vector space $H^{1,0}(X,\mathbb{C})$
of holomorphic one-forms on $X$ so that $\omega_i(a_j)=\delta_{ij}$.
The imaginary parts of $\omega_i$ are linearly independent over 
$\mathbb{R}$ and hence the one-forms 
$\omega_1,\dots,\omega_g$ determine a lattice 
$\Lambda(X)$ in $\mathbb{C}^g$, obtained by integration 
over the geometric symplectic basis 
$a_1,b_1,\dots,a_g,b_g$
of $H_1(S_g,\mathbb{Z})$.  
The quotient of $\mathbb{C}^g$ by this lattice then is the Jacobian
${\cal J}(X)$ of $X$.
The fundamental group $A$ of ${\cal J}(X)$ is isomorphic
to the integral homology group $H_1(S_g,\mathbb{Z})=\mathbb{Z}^{2g}$ of 
$S_g$, and hence using duality provided by the symplectic form, 
to the group $H^1(S_g,\mathbb{Z})$.  

If $X$ varies in a smooth family, then the holomorphic one-forms
$\omega_i=\omega_i(X)$ on $X$ defined by $\omega_i(X_i)(a_j)=\delta_{i,j}$ 
also vary smoothly. 
This means that there exists
a smooth fiber bundle $\Theta:W\to B$ whose fiber over $X$ is just the
Jacobian ${\cal J}(X)$ of $X$. 

The bundle $W$ is naturally a quotient of the \emph{Hodge bundle}, the
complex vector bundle $Z\to B$ whose fiber at a point $X\in B$ equals the complex
vector space of holomorphic one-forms on $X$. This bundle is in general not trivial as a
complex vector bundle. However, it is flat as a real vector bundle with symplectic fiber.
Namely, the action of the mapping class group of $S_g$ on the first cohomology of 
$S_g$ defines a homomorphism $\rho:\Gamma_g\to Sp(2g,\mathbb{Z})$, and the 
Hodge bundle is the bundle
\[Z=\tilde B\times H^1(S_g,\mathbb{R})/\pi_1(B)\]
where the action of $\pi_1(B)$ is defined by 
$(x,Y)g=(xg,\rho(g)^{-1}Y).$

The right action of $H^1(S_g,\mathbb{R})$ by translation commutes with the action by $\rho$ 
and hence there is a quotient bundle $W=Z/H^1(S_g,\mathbb{Z})$ whose fiber
at $X$ just equals the Jacobian $J(X)$ of $X$.

The Jacobian ${\cal J}(X)$ parameterizes divisors of degree $0$ 
on $X$ up to linear
equivalence, that is, up to adding a divisor of a meromorphic function. 
Thus ${\cal J}(X)$ can be viewed as the subgroup of the 
Picard group of $X$ parameterizing holomorphic 
line bundles of degree zero. The group structure is
given by the tensor product, with 
the trivial line bundle as the neutral element.

A \emph{section} of the bundle $\Theta:W\to B$ is a smooth map 
$\sigma:B\to W$ such that $\Theta\circ \sigma={\rm Id}$. 
Such a section 
then determines a splitting of the extension
\begin{equation}\label{splittingsequence}
1\to A\to \pi_1(W)\xrightarrow{\Theta_*} \pi_1(B)\to 1,\end{equation}
that is, for some $x\in B$ it defines a homomorphism
$\sigma_*:\pi_1(B,x)\to 
\pi_1(W,\sigma(x))$
such that $\Theta_* \circ \sigma_*={\rm Id}$.
%

The following is a a topological analog of a well known statement on group
extensions as discussed in Proposition IV.2.1 of \cite{Bro82}. Namely, if $G$ is a 
discrete group and if $A$ is any $G$-module, then $A$-conjugacy classes of splittings of the 
split extension
\begin{equation}\label{split}
1\to A\to A\rtimes G\to G\to 1
\end{equation}
are in 1-1-correspondence with the elements of $H^1(G,A)$.

In the topological setting, a conjugacy class of an element in the fundamental group 
$\pi_1(Y,y)$ of a path connected topological space $Y$ is just a free homotopy class of loops
in $Y$. Being able to move the basepoint continuously is the main difference to the setting of discrete
groups. With this in mind, the next observation gives a topological interpretation of the sequence
(\ref{split}) in our setting. Here the $G$-module $A$ is just the integral cohomology group
$H^1(S_g,\mathbb{Z})$ with the monodromy action of $\pi_1(B)$ defined by the representation $\rho$.

\begin{proposition}\label{homotopysplit}
  Homotopy classes of sections $B\to W$ form a group which is 
isomorphic to 
$H^1(B,H^1(\pi_1(S_g),\mathbb{Z}))$. 
\end{proposition}
\begin{proof} Let $\sigma:B\to W$ be a section. Then 
for some basepoint $x\in B$, the induced
homomorphism $\sigma_*:\pi_1(B,x)\to \pi_1(W,\sigma(x))$ 
defines a splitting of the extension (\ref{splittingsequence}).

Let as before $Z\to B$ be Hodge bundle with fiber
$H^1(S_g,\mathbb{R})$, viewed as an abelian group. Recall that we have
$W=Z/H^1(S_g,\mathbb{Z})$. 
For each $x\in B$, 
there is a natural
action of $H^1(S_g,\mathbb{R})$ on the fiber $W_x$ of $W$ over $x$.
We claim that if $\eta$ is another smooth section of
$W$ then $\sigma$ and $\eta$ are homotopic if and only if
there exists a smooth section $\rho$ of the bundle $Z$ so 
that $\eta=\rho(\sigma)$, where the action of $\rho$
is fiber preserving.

Namely, if $\rho$ is any section of $Z$, then using the fiberwise
group structure (or, alternatively, the fact that the fiber of $Z$
is contractible),
there is a smooth fiber preserving homotopy $h_t$ of $\rho=h_1$ to the section
$h_0$ of $Z$ 
which associates to $x\in B$ the neutral element in $H^1(S_g,\mathbb{R})=Z_x$.
Then $s\to h_s\sigma$ is a fiber preserving homotopy between 
$\sigma$ and $\rho\sigma$.

On the other hand, let us assume that $\eta$ is homotopic to $\sigma$.
Let $h_t$ be a fiber preserving homotopy connecting $h_0=\sigma$ to 
$h_1=\eta$. 
Choose a point $x\in B$ and a preimage $q\in Z_x$ of $\sigma(x)$ in the fiber $Z_x$ of 
$Z$ at $x$. The path $t\to h(t,x)=h_t(x)$ lifts to a path $\tilde h(t,x)$ 
in $Z_x$ beginning at $q$.
We can write $\tilde h(1,x)=\beta(x)+q$ for some $\beta(x)\in H^1(S_g,\mathbb{R})$ (here we
write the group multiplication additively). Now if $u\in Z_x$ is another preimage of $\sigma(x)$ in 
$Z_x$ then $u=m+q$ for some $m\in A$, identified with the lattice in $H^1(S_g,\mathbb{R})$ defined
by the complex structure on the fiber $E_x$ of $E$. 
The path 
$t\to \tilde h(t,x)+m$ is the lift of $t\to h(t,x)$ through $u$. Thus 
the difference of the endpoints 
$\beta(x)=\tilde h(1,x)-\tilde h(0,x)\in H^1(S_g,\mathbb{R})$ does not depend on the
choice of the preimage $q$ of $\sigma(x)$ and hence only depends on $h$.  
Furthermore, the map $x\to \beta(x)$ is continuous and hence defines a section of 
$Z$ with $\eta=\beta(\sigma)$.
This is what we wanted to show.

Let now $C^\infty(W)$ and $C^\infty(Z)$ 
be the sheaf of smooth sections of $W$ and $Z$, respectively.
Since $W$ has a fiberwise structure of an
abelian group, these are sheaves of abelian groups. Similarly we define the sheaf
$C^\infty(H^1(S_g,\mathbb{Z}))$ of smooth sections of the fiber bundle with fiber
the group $A=H^1(S_g,\mathbb{Z})$ (this is meant to be the twisted bundle). 
We then obtain a short exact sequence of sheaves
\[0\to C^\infty(H^1(S_g,\mathbb{Z}))\to C^\infty(Z)\to C^\infty(W)\to 0.\]

It follows from the above discussion that $H^0(C^\infty(W))/H^0(C^{\infty}(Z))$ can naturally be identified with 
the group of homotopy classes of sections of $W$. On the other hand, 
the short exact sequence of sheaves induces 
a long exact cohomology sequence 
\[ \cdots \to H^0(C^{\infty}(Z))\to H^0(C^{\infty}(W))\to H^1(C^\infty(H^1(S_g,\mathbb{Z}))) \to H^1(C^\infty(Z))\to \cdots.\]

Since the sheaf $C^\infty(Z)$ is the sheaf of smooth sections of a flat vector bundle $Z\to B$, it is fine and hence
acyclic. Namely, a morphism of $C^\infty(Z)$ is a smooth section of the bundle 
$Z^*\otimes Z$ over $B$ whose fiber over $x$ equals the vector space of endomorphisms of $Z_x$, that is, it equals  
the vector space $Z_x^*\otimes Z_x$. This vector space has a distinguished real one-dimensional subspace consisting of
constant multiples of the identity, and these one-dimensional subspaces define a trivial one-dimensional real
subbundle $L$ of $Z^*\otimes Z$. A smooth section of $L$ can be identified with a smooth function on $B$.
The identity morphism corresponds to the real number 1 in this interpretation. 

Now if ${\cal U}=\{U_i\}$ is a locally finite covering of $B$, then there is a subordinate partition of unity, 
and using the identification of the fiber of the bundle $L$ with $\mathbb{R}$, this partition of unity defines a 
partition of unity for the sheaf $C^\infty(Z)$, showing that this sheaf fine and hence acyclic. Thus we obtain the 
short exact sequence
\[H^0(C^\infty(Z))\to H^0(C^\infty(W))\to H^1(C^\infty(H^1(S_g,\mathbb{Z})))\to 0.\]
But $H^1(C^\infty(H^1(S_g,\mathbb{Z}))=H^1(B,H^1(S_g,\mathbb{Z}))$ by de Rham's theorem which completes
the proof of the proposition.
\end{proof}

Let again $\sigma:B\to W$ be a smooth section. Then by Abel's theorem,
for every $x\in B$, the value $\sigma(x)$ of 
$\sigma$ at $x$ can be thought of as a holomorphic line bundle of degree $0$ on the fiber $E_x$ of $E$ at $x$
depending smoothly on $x$. Thus $\sigma$ defines a fiberwise holomorphic line bundle 
$L(\sigma)\to E$. 

For our next observation, let us denote by ${\cal F}$ the sheaf of smooth functions on $E$ whose 
restriction to a fiber is holomorphic, and let ${\cal F}^*$ be the subsheaf of functions which vanish nowhere.
These are sheaves of abelian groups. 

\begin{lemma}\label{sheafpara}
The cohomology group $H^1(E,{\cal F}^*)$ parameterizes smooth complex line bundles on $E$ whose restrictions to 
a fiber are holomorphic.
\end{lemma}
\begin{proof}
A smooth complex line bundle $L$ on $E$ whose restriction to a fiber is holomorphic is defined by some good cover
${\cal U}=\{U_i\mid i\}$ of $E$ with the property that for each $i$, the intersection of $U_i$ with a fiber is a disk or empty, and
smooth trivializations of $L$ on each of the open sets $U_i\in {\cal U}$ whose restrictions to the intersections of $U_i$ with a fiber
are holomorphic.

Then transition functions for $L$ on $U_i\cap U_j$ are smooth $\mathbb{C}^*$-valued functions on $U_i\cap U_j$ whose
restrictions to a fiber are holomorphic. Thus these functions define a one-cocycle for ${\cal U}$ with values in ${\cal F}^*$, and 
then they define a class in $H^1(E,{\cal F}^*)$.

Vice versa, each one-cocycle for ${\cal U}$ with values in ${\cal F}^*$ defines  a smooth fiberwise holomorphic line bundle on $E$
by gluing the trivial bundle over the sets $U_i$ with the transition functions on $U_i\cap U_j$ defined by the cocycle. Passing to 
cohomology yields the lemma.
\end{proof}

Smooth line bundles on 
$E$ are defined by classes in the cohomology group 
$H^1(E,(C^\infty)^*)$ where $C^\infty$ is the sheaf of smooth functions on 
$E$ and $(C^\infty)^*$ is the sheaf of smooth functions vanishing nowhere. 
The short exact sequence
\[\begin{tikzcd}
0\arrow[r] & \mathbb{Z}\arrow[r] & C^\infty\arrow[r,"\exp"] &(C^\infty)^* \arrow[r] &0\end{tikzcd}\]
then induces a long exact sequence in cohomology 
\[\begin{tikzcd} \cdots \arrow[r] &H^1(E,C^\infty)\arrow[r] &H^1(E,(C^\infty)^*)\arrow[r,"\delta"] &H^2(E,\mathbb{Z})\arrow[r]
 & \cdots 
\end{tikzcd}.\]
Since the sheaf $C^\infty$ is fine, this sequence 
describes explicitly the parameterization of the group of isomorphism classes of smooth line 
bundles on $E$ by their Chern classes, that is, by the group $H^2(E,\mathbb{Z})$.

Our next goal is to verify that homotopic sections of the bundle $W$ define smoothly
equivalent line bundles, or, equivalently, line bundles with the same Chern class, and 
that this Chern class is just the cohomology class in $H^1(B,H^1(S_g,\mathbb{Z}))$ corresponding
to this homotopy class by Proposition \ref{homotopysplit}.

To this end note that 
since the second cohomology of $E$ is representable, each class 
$\alpha\in H^2(E,\mathbb{Z})$ is the Chern class of a smooth complex line bundle, obtained as
the pull-back of the tautological line bundle
under a smooth map $f:E\to \mathbb{C}P^N$ 
for some sufficiently large $N$ 
which defines $\alpha$. Homotopic maps define isomorphic line bundles.

Now if $L\to E$ is a smooth complex line bundle, then the \emph{degree} of 
$L$ can be defined as the evaluation of its Chern class $c_1(L)$ on one 
(and hence on any) fiber.  Denote as before 
by $c_1(\nu)$ the Chern class of the vertical cotangent bundle. 
By Proposition \ref{surfacebundles}, the subgroup $H^1(B,H^1(S_g,\mathbb{Z}))$ of 
$H^2(E,\mathbb{Z})$ is contained in the kernel of the homomorphism
$\alpha\to \alpha\cup c_1(\nu)$. The restriction of this homomorphism 
to $\Pi^*H^2(B,\mathbb{Z})$ is injective. 
The next proposition provides the 
connection between the constructions in this section.

\begin{proposition}\label{parametrize}
Let $E\to B$ be a surface bundle over a surface. Then 
the cohomology group $H^1(B,H^1(S_g,\mathbb{Z}))\subset H^2(E,\mathbb{Z})$
parameterizes isomorphism classes of 
fiberwise holomorphic 
line bundles $L$ on $E$ of degree $0$ whose Chern class $c_1(L)$ satisfies
$c_1(L)\cup c_1(\nu)=0$. 
\end{proposition}
\begin{proof} By the above discussion, 
a smooth section $\sigma$ of the bundle $W$ defines
on the one hand an equivalence class of a complex line bundle 
$L(\sigma)$ on $E$ 
whose restrictions to a fiber is holomorphic of degree 0. On the other hand, by
Proposition \ref{homotopysplit} and Proposition \ref{surfacebundles}, 
it defines a cohomology class in
$H^1(B,H^1(S_g,\mathbb{Z}))\subset H^2(E,\mathbb{Z})$.
We have to show that this cohomology class is just
the first Chern class of $L(\sigma)$. As smooth line bundles on $E$ with the same
Chern class are equivalent, this implies that homotopic sections of $W$ define
smoothly equivalent line bundles on $E$, a fact which can also be verified 
directly.

Consider again the sheaf ${\cal F}$ of smooth functions on 
$E$ which are fiberwise holomorphic, the subsheaf
${\cal F}^*$ of functions in ${\cal F}$ which vanish
nowhere, and the sheaf $\mathbb{Z}$ of 
locally constant $\mathbb{Z}$-valued functions.

The short exact sequence
\begin{equation}\label{holomorphic}\begin{tikzcd}
0\arrow[r] & \mathbb{Z}\arrow[r] & {\cal F}\arrow[r,"\exp"] & {\cal F}^* \arrow[r] &0\end{tikzcd}\end{equation}
%
induces a long exact sequence in cohomology. 
Since the sheaf $C^\infty$ of smooth functions on $E$ is fine, 
the inclusions ${\cal F}\to C^\infty$ and ${\cal F}^*\to (C^\infty)^*$ then determine 
an exact commutative diagram 
\begin{equation}\label{diagram}\begin{tikzcd}
\cdots \arrow[r,"\exp"] &H^1(E,{\cal F}^*) \arrow[r]\arrow[d,"\eta"] & H^2(E, \mathbb{Z}) \arrow[r]\arrow[d,"{\rm Id}"] &
H^2(E,{\cal F}) \arrow[r]\arrow[d] & \cdots \\
\cdots \arrow[r,"\exp"] & H^1(E,(C^\infty)^*) 
\arrow[r,"\delta"] & H^2(E,\mathbb{Z})\arrow[r] & 0 \arrow[r] &\cdots
\end{tikzcd} \end{equation}

We claim that the homomorphism $\eta$ is surjective. By exactness and since the diagram commutes, this
follows if we can show that $H^2(E,{\cal F})=0$. However, if $E_x$ is any fiber of $E$ then
we have $H^2(E_x,{\cal O})=0$ where as usual, ${\cal O}$ is the sheaf of holomorphic functions on $E_x$.
Namely,  by Serre duality, this cohomology group can be identified with the space of holomorphic
two-forms on $E_x$, and this space vanishes since the complex dimension of $E_x$ equals one.
On the other hand, the sheaf of smooth functions on $B$ is fine and hence the Leray spectral sequence 
shows that indeed,  $H^2(E,{\cal F})=H^0(B,H^2(E_x,{\cal O}))=0$.

Since $H^1(E,C^\infty)=0$, the 
homomorphism $\delta$ is an isomorphism. This yields that every smooth line bundle
on $E$ is smoothly equivalent to a line bundle whose restriction to a fiber is holomorphic. 
It also follows that 
the homomorphism $\eta$ maps $H^1(E,{\cal F}^*)/\exp(H^1(E,{\cal F}))$ isomorphically onto 
$H^1(E,(C^\infty)^*)$. Hence associating to an element in this group its Chern class is an isomorphism.

On the other hand, for each $x\in B$ the vector space 
$H^1(E_x,{\cal O})$ is just the space of 
holomorphic one-forms on the fiber $E_x$ of $E$ over $x$ by Serre duality. 
That is, $H^1(E_x,{\cal F})$ is the fiber at $x$ of the Hodge bundle $Z\to B$.
Thus using the fact that the sheaf of smooth sections of $Z$ is fine (see the discussion in the proof of 
Proposition \ref{homotopysplit} for details),  
the Leray spectral sequence shows that $H^1(E,{\cal F})=H^0(Z)$, the vector space of smooth sections of $Z$.

Now consider the part 
\begin{equation}\label{diagram}\begin{tikzcd}
\cdots \arrow[r] &H^1(E,\mathbb{Z}) \arrow[r,"exp"]\arrow[d,"{\rm Id}"] & H^1(E,{\cal F}) \arrow[r]\arrow[d] &
H^1(E,{\cal F}^*) \arrow[r]\arrow[d,"\eta"] & \cdots \\
\cdots \arrow[r] & H^1(E,\mathbb{Z}) 
\arrow[r] & 0\arrow[r] & H^1(E,(C^\infty)^*) \arrow[r] &\cdots
\end{tikzcd} \end{equation}
of the above exact diagram. It shows that the kernel of $\eta$ can be identified with $H^0(Z)/\exp H^1(E,\mathbb{Z})$. 
By Proposition \ref{homotopysplit} and its proof, this subgroup is precisely the group of sections of 
the bundle $W$ which are 
homotopic to the trivial section. As a consequence, homotopic sections of $W$ define line bundles
with the same Chern class and hence line bundles which are smoothly equivalent. 
Furthermore, associating to a homotopy class of a section of $W$ 
the Chern class of the line bundle it defines is an isomorphism of the space of 
all homotopy classes of sections onto $H^1(B,H^1(S_g,\mathbb{Z}))\subset H^2(E,\mathbb{Z})$.
\end{proof}

\begin{remark}
The assumption that $E\to B$ is a surface bundle over a surface was only used through
the conclusion of Proposition \ref{surfacebundles}.
\end{remark}

\begin{remark}
The proof of Proposition \ref{parametrize} also shows that any smooth complex line
bundle on $E$ is smoothly equivalent to a line bundle whose restriction to each 
fiber is holomorphic. 
\end{remark}

We use similar ideas to show

\begin{proposition}\label{linebundle}
Let $E\to B$ be a surface bundle which admits a section.
Then there exists a cohomology class $e\in H^2(E,\mathbb{Z})$ with 
$(2g-2)e=c_1(\nu)$, and for all $q$ we have
\[H^q(E,\mathbb{Z})= H^q(B,\mathbb{Z})\oplus
  H^{q-1}(B,H^1(S_g,\mathbb{Z}))\oplus 
eH^{q-2}(B,\mathbb{Z}).\]
\end{proposition}
\begin{proof}
Let $f:B\to E$ be a section of the surface bundle $\Pi:E\to B$. 
Assume as before that each fiber $E_x$ of $E$ is equipped with
a complex structure varying smoothly with $x$. Then for each $x\in B$, the point $f(x)\in E_x$
can be thought of as a divisor in $E_x$ defining a complex line bundle $L_x$ of degree $1$ on $E_x$.
As these line bundles depend smoothly on $x$, they fit together to a fiberwise holomorphic line
bundle $L$ of fiberwise degree one. 

Let $c_1(L)\in H^2(E,\mathbb{Z})$ be the Chern class of $L$.
Consider the inclusion $\iota:E_x\to E$. As the fiberwise degree of $L$ equals one, 
we know that $\iota^*c_1(L)$ is a generator of $H^2(E_x,\mathbb{Z})$. 
Thus the spectral sequence argument in the proof of 
Proposition 3.1 of \cite{Mo87} applies to compute the cohomology of $E$ with coefficients in $\mathbb{Z}$
(this argument only uses surjectivity of $\iota^*$
for the  coefficient ring under consideration), yielding the formula in 
Proposition \ref{cohomologyprop} but with coefficients $\mathbb{Z}$. 
\end{proof}

\begin{remark} Although the existence of a
  section for a surface bundle $E\to B$ 
is simply equivalent to stating that the induced homomorphism 
$\pi_1(B)\to \Gamma_g$ lifts to a homomorphism $\pi_1(B)\to \Gamma_{g,1}$,  
we do not know how to characterize this property in purely topological terms of the 
surface bundle. In fact, if $E\to B$ is a surface bundle over a surface, then 
$E$ is bordant to a surface bundle over a surface which admits a section, see \cite{H83}. 
\end{remark} 

Proposition \ref{linebundle} describes a 
correspondence between line bundles on a surface
bundle over a surface $\Pi:E\to B$, their Chern classes and their Poincar\'e dual. This can be extended as follows. 
Namely, a section $f:B\to E$ can be thought of as a section in the bundle over $B$ whose fiber consists
of all effective divisors of degree $1$ on the fiber of $E$. This viewpoint generalizes as follows.

An effective divisor on a Riemann surface
$X$ of degree $k\geq 1$ is just a weighted collection of points on $X$ with positive weights
which sum up to $k$. Thus
 there is a natural
topology on the total space ${\cal D}^k$ of all effective 
divisors of degree $k$ on the fibers of $E$ 
defined as follows. Let $V\to B$ be the fiber product of 
$k$ copies of the fiber of $E$. There is a natural smooth fiber preserving free action of 
the symmetric group in $k$ variables on $V$. Then ${\cal D}^k$ can
be identified with the quotient of this action and hence it inherits from
$V$ the quotient topology. 
By abuse of notation, we denote again by $\Pi$ the projection
${\cal D}^k\to B$. 

Let us assume that there exists a section 
$\psi:B\to {\cal D}^k$. Associate to this section the fiberwise holomorphic line 
bundle $L(\psi)$ whose restriction of a fiber $E_x$ is dual to the divisor $\psi(x)$, 
and associate to $L(\psi)$ its Chern
class $c_1(L(\psi))\in H^2(E,\mathbb{Z})$.

Now the section $\psi$ of ${\cal D}^k$ defines a cycle in $E$ 
which can be seen as follows. The projection of the fat diagonal of $V$ is 
submanifold $N$ of ${\cal D}^k$ of fiberwise
positive real codimension $2$. Thus by transversality, 
we may assume that $\psi$ is transverse to this submanifold. 
Then there are (at most) finitely many points $x_1,\dots,x_m$ such that
$\psi(x_i)\in N$, and for each $i$,  the image of $x_i$ in $E_{x_i}$ consists of 
$m-1$ distinct points, with precisely one point of multiplicity $2$.

Choose a triangulation
of $B$ containing the points $x_i$ as vertices. For any point $x\not\in \{x_1,\dots, x_m\}$, 
the preimage of $x$ in $E_x$ defined by $\psi$ (that is, the union of all points of $\psi(x)$)
consists of precisely $m$ points moving smoothly
with the base. Thus each two-simplex
of the triangulation has precisely $k$ preimages in $E$, and the same holds true for 
all one-simplices. The preimages of the points $x_i$ consist of only $m-1$ points. 
It follows from this construction that the union of these triangles is a surface $\Sigma\subset E$. 
The orientation of $B$ induces an orientation of $\Sigma$. The restriction of the projection 
$\Pi$ to $\Sigma$ is a branched cover, ramified precisely at the points $x_i$.
Thus $\Sigma$ defines a homology class $\beta(\psi)\in H_2(E,\mathbb{Z})$. 
We have

\begin{proposition}\label{poincaredual}
The class $\beta(\psi)$ is Poincar\'e dual to $c_1(L(\psi))$.
\end{proposition}
\begin{proof} Let us recall how to construct from the embedded surface $\Sigma$ which is transverse to the 
fibers of $E$ a line bundle whose restriction to a fiber is holomorphic. 
Namely, for a point $x\in \Sigma$, choose a neighborhood $U$ of $x$ in $E$ so that
$U\cap \Sigma$ is a smooth disk. There exists a smooth $\mathbb{C}$-valued 
function $f$ on $U$ whose restriction to a fiber is  holomorphic and 
with nowhere vanishing derivative, so that $U\cap \Sigma$ is the level set of level zero. 
Choose a covering of $\Sigma$ by such sets, with corresponding functions. 
On the intersections of these sets, the quotients of these functions do not vanish. Thus these
functions define a cocycle whose cohomology class defines a line bundle. This
line bundle has a smooth section which is fiberwise holomorphic and vanishes precisely on $\Delta$.
In particular, this line bundle coincides with the line bundle $L(\psi)$.

Now if the section $\psi$ intersects the fat diagonal $N$ of ${\cal D}^k$, then at the finitely many intersection
points with $N$, choose the function so that it has a double zero at that point and proceed as before.  

Since every class in $H_2(E,\mathbb{Z})$ can be represented by a 
smooth map $f:M\to E$ where $M$ is a closed oriented surface of some genus $h\geq 0$, for the proof of the proposition it
now suffices to show the following. 

Assume without loss of generality that $f(M)$ intersects $\Sigma$ transversely in finitely many points
$y_1,\dots,y_p\in E-\cup_j\Pi^{-1}(x_j)$, with intersection index $\sigma(y_i)\in \pm 1$. We have to show that 
$c_1(L(\psi))(f(M))=\sum_i\sigma(y_i)$.

As the line bundle $L(\psi)$ is trivial on 
$E-\Sigma$,
the pull-back of $L(\psi)$ under $f$ is a complex line bundle on $M$ with a section 
which vanishes precisely to first order at the points $y_i$, and the index of this zero is $\pm 1$ depending on
whether the intersection is positive or negative.
On the other hand,  $c_1(f^*L(\psi))(M)$ equals the number of 
zeros of a section of $f^*L(\psi)$, counted with sign and multiplicities,
provided that this
section is transverse to the zero section.  
Together this means that 
$c_1(f^*L(\psi)))=\sum_i\sigma(y_i)$. Since 
$f$ was an aribtrary map of a closed oriented surface $M$ into $E$, we conclude that indeed, for any second homology 
class $\alpha\in H_2(E,\mathbb{Z)}$ we have $\alpha\cdot \Sigma=c_1(L(\psi))(\alpha)$ as claimed.    
\end{proof}

\bigskip
\noindent
MATHEMATISCHES INSTITUT DER UNIVERSIT\"AT BONN, 
ENDENICHER ALLEE 60, 
53115 BONN, 
GERMANY

\smallskip
\noindent 
e-mail: ursula@math.uni-bonn.de
\end{document}